\documentclass[12pt,reqno]{amsart}%
\usepackage{amsmath, color}
\usepackage[margin=1.5in]{geometry}
\usepackage{graphicx}
\usepackage{amsfonts}
\usepackage{amssymb}%
\usepackage[utf8]{inputenc}
\usepackage{comment}
\usepackage{hyperref}
\usepackage{mathtools}

\newtheorem{theorem}{Theorem}[section]
\theoremstyle{plain}

\newtheorem{corollary}{Corollary}[section]

\newtheorem{lemma}{Lemma}[section]

\newtheorem{proposition}{Proposition}[section]

\numberwithin{equation}{section}
\theoremstyle{definition}

\theoremstyle{remark}
\newtheorem{remark}{Remark}[section]
\allowdisplaybreaks

\def\pd{\partial}
\def\re{\mathbb{R}}

\def\ti{\tilde}
\def\mbb{\mathbb}

\newcommand{\eqal}[1]{\begin{equation}\begin{aligned}#1\end{aligned}\end{equation}}

\newcommand{\osc}{\mathrm{osc}}

\newcommand{\Dg}{\Delta_g}
\newcommand{\dg}{\nabla_g}

\title[Singularities of the LMCF]{A priori estimates for Singularities of the Lagrangian Mean Curvature Flow with supercritical phase
}

\author{Arunima Bhattacharya and Jeremy Wall}

\address{Department of Mathematics, Phillips Hall\\
 the University of North Carolina at Chapel Hill, NC }
\email{arunimab@unc.edu}

\address{Department of Mathematics, Phillips Hall\\
 the University of North Carolina at Chapel Hill, NC }
\email{jwall2@unc.edu}

\begin{document}


\begin{abstract}
In this paper, we prove interior a priori estimates for singularities of the Lagrangian mean curvature flow assuming the Lagrangian phase is supercritical. We prove a Jacobi inequality that holds good when the Lagrangian phase is critical and supercritical. We further extend our results to a broader class of Lagrangian mean curvature type equations.

\end{abstract}

\maketitle

\section{Introduction}
A family of Lagrangian submanifolds $X(x,t):\re^n\times\re\to\mbb C^n$ evolves by \textit{Lagrangian mean curvature flow} if it solves
\eqal{
\label{LMCF}
(X_t)^\bot=\Delta_gX=\vec H,
} 
where $\vec H $ denotes the mean curvature vector of the Lagrangian submanifold. After a change of co-ordinates, one can locally write $X(x,t)=(x,Du(x,t))$ such that $\Delta_gX=J\nabla_g\Theta$ (see \cite{HL})  where $\Theta$ is the Lagrangian angle given by
\begin{equation}
    \Theta=\sum_{i=1}^n\arctan\lambda_i \label{sl}
\end{equation}
with $\lambda_i$ denoting the eigenvalues of the Hessian $D^2u$. Here $g=I_n+(D^2u)^2$ is the induced metric on $(x,Du(x))$ and $J$ is the almost complex structure on $\mathbb C^n$. This
results in a local potential $u(x,t)$ evolving by the parabolic equation
\begin{align*}
&u_t=\Theta,\\
&u(x,0):=u(x).
\end{align*}


Symmetry reductions of the Lagrangian mean curvature flow reduce the above local parabolic representation to an elliptic equation for $u(x)$, which models singularities of the mean curvature flow (see Chau-Chen-He \cite{CCH}). If $u(x)$ solves
\eqal{
\label{s}
\sum_{i=1}^n\arctan\lambda_i=s_1+s_2(x\cdot Du(x)-2u(x)),
}
then $X(x,t)=\sqrt{1-2s_2t}\,(x,Du(x))$ is a \textit{shrinker} or \textit{expander} solution of \eqref{LMCF}, if $s_2>0$ or $s_2<0$ respectively.  If $u(x)$ solves
\eqal{
\label{tran}
\sum_{i=1}^n\arctan\lambda_i=t_1+t_2\cdot x+t_3\cdot Du(x),
}
then $X(x,t)=(x,Du(x))+t(-t_3,t_2)$ is a \textit{translator} solution of \eqref{LMCF}. If $u(x)$ solves
\eqal{
\label{rotator}
\sum_{i=1}^n\arctan\lambda_i=r_1+\frac{r_2}{2}(|x|^2+|Du(x)|^2),
}
then $X(x,t)=\exp(r_2tJ)(x,Du(x))$ is a \textit{rotator} solution of \eqref{LMCF}. A broader class of equations  of interest that generalize equations \eqref{s}, \eqref{tran}, \eqref{rotator}, among others, are the 
\textit{Lagrangian mean curvature type equations}
\eqal{
\label{slag}
\sum_{i=1}^n\arctan\lambda_i=\Theta(x,u(x),Du(x)).
}
See \cite{Y20, BS2} for a detailed discussion.

\noindent \textbf{Notations. }Before we present our main results, we clarify some terminology.
\begin{itemize}
\item[I.] By $B_R$ we denote a ball of radius $R$ centered at the origin.
\item[II.] We denote the oscillation of $u$ in $B_R$ by $\osc_{B_R}(u)$.
\item[III.] Let $\Gamma_R = B_R\times u(B_R)\times Du(B_R)\subset B_R\times\re\times\re^n$. Let $\nu_1,\nu_2$ be constants such that for $\Theta(x,z,p)$, we have the following structure conditions
\begin{align}
    |\Theta_x|,|\Theta_z|,|\Theta_p|&\leq \nu_1,\label{struct}\\
    |\Theta_{xx}|,|\Theta_{xz}|,|\Theta_{xp}|,|\Theta_{zz}|,|\Theta_{zp}| &\leq \nu_2 \nonumber
\end{align}
for all $(x,z,p)\in\Gamma_R$. In the above partial derivatives, the variables $x,z,p$ are treated as independent of each other. Observe that this indicates that the above partial derivatives do not have any $D^2u$ or $D^3u$ terms.
 \end{itemize}
\medskip

Our main result is the following: we prove interior Hessian estimates for shrinkers, expanders, translators, and rotators of the Lagrangian mean curvature flow and further extend these results to the broader class of Lagrangian mean curvature type equations, under the assumption that the Lagrangian phase is supercritical, i.e. $|\Theta|\geq (n-2)\frac{\pi}{2}+\delta$.

\begin{theorem}\label{main0}
If $u$ is a smooth solution of any of these equations: \eqref{s}, \eqref{tran}, and \eqref{rotator} on $B_{R}(0)\subset \mathbb{R}^{n}$ where  $n\geq 2$ and $|\Theta|\geq (n-2)\frac{\pi}{2}+ \delta$, then we have 
\[
|D^2u(0)|\leq \exp\big[C_1\csc^{9n-6}(\delta/2)\big]\exp\bigg[C_2\csc^{9n-6}(\delta/2)\big(\osc_{B_{R}}(u)/R^2\big)^{4n-2}\bigg]
\]
where $C_1$ and $C_2$ are positive constants depending on $n$ and the following: $s_2$ for \eqref{s}, $t_2,t_3$ for \eqref{tran}, and $r_2$ for \eqref{rotator}.
\end{theorem}

In the more general case we have the following result.
\begin{theorem}\label{main1}
If $u$ is a smooth solution of \eqref{slag} on $B_R(0)\subset\re^n$, with $n\geq 2$ and $|\Theta|\geq (n-2)\frac{\pi}{2}+\delta$ where $\Theta(x,z,p)\in C^2(\Gamma_R)$ satisfies \eqref{struct} and is partially convex in $p$, then we have
\[
|D^2u(0)|\leq \exp\big[C_1\csc^{9n-6}(\delta/2)\big]\exp\bigg[C_2\csc^{9n-6}(\delta/2)\big(\osc_{B_{R}}(u)/R^2\big)^{4n-2}\bigg]
\]
where $C_1$ and $C_2$ are positive constants depending on $\nu_1,\nu_2$, and $n$.
\end{theorem}

   \begin{remark}
    We would like to point out that the Hessian estimates in terms of the gradient of the potential hold good for \eqref{s}, \eqref{tran}, \eqref{rotator} when the phase is critical and supercritical $|\Theta|\geq (n-2)\frac{\pi}{2}$, if one uses a modification of the integral approach presented in \cite{AB}. The estimates also extend to \eqref{slag} when $\Theta$ is partially convex in the gradient variable. However, deriving the corresponding gradient estimates is more subtle, and we address this in our forthcoming paper \cite{BWall3}.
   \end{remark}

An application of the above results is that $C^0$ viscosity solutions to \eqref{s},\eqref{tran}, and \eqref{rotator} with $|\Theta|\geq (n-2)\frac{\pi}{2}+\delta$ are analytic inside the domain of the solution.

For solutions of the special Lagrangian equation, i.e. $\Theta=constant$ with critical and supercritical phase $|\Theta|\geq (n-2)\frac{\pi}{2}$, Hessian estimates were obtained by Warren-Yuan \cite{WY9,WY}, Wang-Yuan \cite{WaY}, Li \cite{Lcomp} via a compactness approach, Shankar \cite{shankar2024hessian} via a doubling approach, and Zhou \cite{ZhouHess} for estimates requiring Hessian constraints which generalize criticality. The singular $C^{1,\alpha}$ solutions to \eqref{sl} constructed by Nadirashvili-Vl\u{a}du\c{t} \cite{NV} and Wang-Yuan \cite{WdY} show that interior regularity is not possible for subcritical phases $|\Theta|<(n-2)\frac{\pi}{2}$, without an additional convexity condition, as shown in Bao-Chen \cite{BCconvex}, Chen-Warren-Yuan \cite{CWY}, and Chen-Shankar-Yuan \cite{CSY}, and that the Dirichlet problem is not classically solvable for arbitrary smooth boundary data. For solutions of the Lagrangian mean curvature equation, i.e. $\Theta=\Theta(x)$ is variable, Hessian estimates for convex smooth solutions with $\Theta\in C^{1,1}$ were obtained by Warren in \cite[Theorem 8]{WTh}. For $C^{1,1}$ critical and supercritical phase, interior Hessian and gradient estimates were established by Bhattacharya \cite{AB, AB2d} and Bhattacharya-Mooney-Shankar \cite{BMS} (for $C^2$ phase) respectively. See also Lu \cite{Siyuan}. Interior Hessian estimates for supercritical $C^{0,1}$ phase were derived by Zhou \cite{Zhou1}. Recently interior Hessian estimates for supercritical and critical $C^{0,1}$ phases were found by Ding \cite{Ding1}. For convex viscosity solutions, interior regularity was established for $C^2$ phase by Bhattacharya-Shankar in \cite{BS1} and optimal regularity conditions were derived in \cite{BS2}. If $\Theta$ is merely in $C^{\alpha}$ and supercritical, counterexamples to Hessian estimates exist as shown in \cite{AB1}. For more work on the special Lagrangian and Lagrangian mean curvature type equations, we refer the reader to \cite{CaoWang,WHB24,QZ24,LiuBao}.

When the Lagrangian phase depends on both the potential and the gradient of the potential of the Lagrangian submanifold, $\Theta(x,u,Du)$, less is known. In \cite{BW}, the authors proved Hessian estimates for solutions of \eqref{s}, \eqref{tran}, \eqref{rotator}, \eqref{slag} under the assumption that the phase is hypercritical, i.e. $|\Theta|\geq (n-1)\frac{\pi}{2}$, a condition which results in the convexity of the potential function $u$. Note that for solutions of \eqref{slag}, Hessian estimates do not hold good without the additional convexity assumption of $\Theta$ in the gradient variable, $Du$, as illustrated by the counterexamples constructed in \cite{BS2}. The concavity of the arctangent operator in \eqref{sl} is closely associated with the range of the Lagrangian phase. When $|\Theta|\geq (n-1)\frac{\pi}{2}$, then $\lambda_i>0$ for all $1\leq i\leq n$, making the arctangent operator concave. However, when the phase is critical and supercritical, i.e. $|\Theta|\geq (n-2)\frac{\pi}{2}$, the potential function $u$ lacks convexity unlike in \cite{BW}. In this range, the level set $\{ \lambda \in \mathbb{R}^n \vert \lambda$ satisfying $ \eqref{sl}\}$ is convex \cite[Lemma 2.2]{YY} but the Hessian of the potential $u$ has no lower bound. 
At the critical value, the Hessian $D^2u$
can possess negative eigenvalues, with the smallest potentially approaching $-\infty$: This makes deriving the Jacobi inequality a challenging problem. In Lemma \ref{ptJ}, we derive a Jacobi inequality for a suitable choice of the slope of the Lagrangian graph that holds good when the phase is critical and supercritical. When the phase is supercritical, i.e. $|\Theta|\geq (n-2)\frac{\pi}{2}+\delta$ with $\delta>0$, the Hessian $D^2u$ has a lower bound given by $-\cot\delta I_n$, which we exploit to perform a version of the Lewy-Yuan rotation \cite{CW2} originally introduced in \cite{YY}. This enables us to derive both a mean value property and a Sobolev inequality for the slope of the Lagrangian graph. The Jacobi inequality combined with the mean value property and the Sobolev inequality leads to the desired estimate. 

Note that the results in section \ref{AJI} are valid when $\Theta$ is critical and supercritical.
However, the remainder of the proof in this paper does not hold at the critical value and requires $\Theta$
to be supercritical. The convexity of the level set at the critical value is insufficient to perform the Lewy-Yuan rotation \cite{YY} necessary to derive an appropriate mean value property.

\section*{Funding} AB acknowledges
the support of NSF Grant DMS-2350290, the Simons Foundation grant MPS-TSM-00002933. JW acknowledges
the support of the National Science Foundation RTG DMS-2135998 grant.

\section{Preliminaries}

For the convenience of the readers, we recall some preliminary results. We first introduce some notations that will be used in this paper.
The induced Riemannian metric on the Lagrangian submanifold $X=(x,Du(x))\subset \mathbb{R}^n\times\mathbb{R}^n$ is given by
\[g=I_n+(D^2u)^2 .
\]
We denote
 \begin{align*} 
    \partial_i=\frac{\partial}{\partial x_i} \text{ , }
     \partial_{ij}=\frac{\partial^2}{\partial x_i\partial x_j} \text{ , }
     u_i=\partial_iu \text{ , }
    u_{ij}=\partial_{ij}u.
    \end{align*}
  Note that for the functions defined below, the subscripts on the left do not represent partial derivatives\begin{align*}
    h_{ijk}=\sqrt{g^{ii}}\sqrt{g^{jj}}\sqrt{g^{kk}}u_{ijk},\quad
    g^{ii}=\frac{1}{1+\lambda_i^2}.
    \end{align*}
Here $(g^{ij})$ is the inverse of the matrix $g$ and $h_{ijk}$ denotes the second fundamental form when the Hessian of $u$ is diagonalized.
The volume form, gradient, and inner product with respect to the metric $g$ are given by
\begin{align*}
    dv_g=\sqrt{\det g}dx &= Vdx \text{ , }\qquad
    \nabla_g v=g^{ij}v_iX_j,\\
    \langle\nabla_gv,\nabla_g w\rangle_g &=g^{ij}v_iw_j \text{ , }\quad
    |\nabla_gv|^2=\langle\nabla_gv,\nabla_g v\rangle_g.
\end{align*}

We state the following lemma. 
\begin{lemma}\label{y1}
		Suppose that the ordered real numbers $\lambda_{1}\geq \lambda_{2}\geq...\geq \lambda_{n}$ satisfy \eqref{sl} with $\Theta\geq (n-2)\frac{\pi}{2}$.
		Then we have \begin{enumerate}
			\item $\lambda_{1}\geq \lambda_{2}\geq...\geq \lambda_{n-1}>0,\quad \lambda_{n-1}\geq |\lambda_{n}|$.
			\item $\lambda_{1}+(n-1)\lambda_{n}\geq 0$.
			\item $\sigma_{k}(\lambda_{1},...,\lambda_{n})\geq 0$ for all $1\leq k\leq n-1$ and $n\geq 2$.
			\item  If $\Theta\geq (n-2)\frac{\pi}{2}+\delta$, then $D^2u\geq -\cot \delta I_n$.
			
		\end{enumerate}
		
	\end{lemma}
\begin{proof}
Properties (1), (2), and (3) follow from \cite[Lemma 2.2]{WaY}. 
We will prove property (4). We claim that if $\Theta\geq (n-2)\frac{\pi}{2}+\delta$, then \[\arctan \lambda_i\geq (\delta-\frac{\pi}{2}). 
\] Suppose not and we have $\arctan \lambda_i<(\delta-\frac{\pi}{2})$ instead. Then we must have 
\[\sum_{k\neq i}\arctan\lambda_k>(n-2)\frac{\pi}{2}+\delta-(\delta-\frac{\pi}{2})=(n-1)\frac{\pi}{2}
\]
which contradicts the fact $\arctan\lambda_k\leq\frac{\pi}{2}.$ So it follows that $D^2u\geq \tan(\delta-\frac{\pi}{2})I_n$.
\end{proof}
For the convenience of the reader, we recall the following proposition from \cite{AB}.

\begin{proposition} \cite[Proposition 4.1]{AB}
    Let $u$ be a smooth solution to \eqref{slag} in $\re^n$. Suppose that the Hessian $D^2u$ is diagonalized and the eigenvalue $\lambda_\gamma$ is distinct from all other eigenvalues of $D^2u$ at point $x_0$. Then at $x_0$ we have
    \begin{equation}\label{dgln}
        \left|\dg \ln\sqrt{1 + \lambda_\gamma^2}\right|^2 = \sum_{k=1}^n\lambda_\gamma^2h_{\gamma\gamma k}^2
    \end{equation}
    and
    \begin{align}
        \Dg \ln\sqrt{1 + \lambda_\gamma^2} &= (1 + \lambda_\gamma^2)h_{\gamma\gamma\gamma}^2 + \sum_{k\neq \gamma}\left(\frac{2\lambda_\gamma}{\lambda_\gamma - \lambda_k} + \frac{2\lambda_\gamma^2\lambda_k}{\lambda_\gamma - \lambda_k} \right)h_{kk\gamma}^2\nonumber\\
        &+\sum_{k \neq \gamma}\left[1 + \frac{2\lambda_\gamma}{\lambda_\gamma-\lambda_k} + \frac{\lambda_\gamma^2(\lambda_\gamma + \lambda_k)}{\lambda_\gamma - \lambda_k}\right]h^2_{\gamma\gamma k}\label{Dgln}\\
        &+\sum_{\substack{k>j \\ k,j\neq \gamma}} 2\lambda_\gamma\left[\frac{1 + \lambda_k^2}{\lambda_\gamma - \lambda_k}+\frac{1 + \lambda_j^2}{\lambda_\gamma - \lambda_j} + (\lambda_j + \lambda_k)\right]h_{kj\gamma}^2\nonumber\\
        &+ \frac{\lambda_\gamma}{1 + \lambda_\gamma^2}\pd^2_{\gamma\gamma}\Theta - \sum_{a=1}^n\lambda_ag^{aa}(\pd_a\Theta)\pd_a\ln\sqrt{1 + \lambda_\gamma^2}.\label{Dgln2}
    \end{align} 
\end{proposition}

\section{Jacobi inequality}\label{AJI}

In this section, we prove a Jacobi-type inequality for the slope of the gradient graph $(x,Du(x))$. 
\begin{lemma}\label{ptJ}
    Let $u$ be a smooth solution to \eqref{slag} in $\re^n$ with $n\geq 3$ and $|\Theta| \geq (n-2)\frac{\pi}{2}$ where $\Theta(x,z,p)$ is partially convex in the $p$ variable. Suppose that the ordered eigenvalues $\lambda_1\geq \lambda_2 \geq \cdots \geq \lambda_n$ of the Hessian $D^2u$ satisfy $\lambda_1 = \cdots =  \lambda_m > \lambda_{m+1}$ at a point $x_0$. Then the function $ b_m = \frac{1}{m}\sum_{i = 1}^m \ln\sqrt{1 + \lambda_i^2}$ is smooth near $x_0$ and satisfies at $x_0$
    \begin{equation}\label{ptJI}
        \Dg b_m \geq c(n)|\dg b_m|^2 - C(\nu_1,\nu_2,n)(1 + |Du(x_0)|^2 ).
    \end{equation}
\end{lemma}

\begin{proof}
    \begin{itemize}
        \item[Step 1.] The function $b_m$ is symmetric in $\lambda_1,\dots,\lambda_m$. Thus, for $m< n$, $b_m$ is smooth in terms of the matrix entries when $\lambda_m>\lambda_{m+1}$. It is still smooth in terms of $x$ since $D^2u(x)$ is smooth, in particular at $x_0$ where $\lambda_1= \dots = \lambda_m>\lambda_{m+1}$. For $m=n$ it is clear that $b_n$ is smooth everywhere. 

        First, we assume that the first $m$ eigenvalues are distinct. Summing up all three lines of \eqref{Dgln} and \eqref{Dgln2}, and grouping the mean curvature terms into $h_{***},h_{**\%},h_{*\%!}$ we get
        \begin{align*}
            m\Delta_g & b_m(x_0) = \sum_{k=1}^m(1 + \lambda_k^2)h_{kkk}^2 + \left(\sum_{i<k\leq m} + \sum_{k<i\leq m}\right)(3 + \lambda_i^2 + 2\lambda_i\lambda_k)h_{iik}^2\\
            &+\sum_{k\leq m < i}\frac{2\lambda_k(1+\lambda_k\lambda_i)}{\lambda_k-\lambda_i}h_{iik}^2 + \sum_{i\leq m < k}\frac{3\lambda_i-\lambda_k + \lambda_i^2(\lambda_i + \lambda_k)}{\lambda_i-\lambda_k}h_{iik}^2\\
            &+2\bigg[\sum_{i<j<k\leq m}(3 + \lambda_i\lambda_j + \lambda_j\lambda_k + \lambda_k\lambda_i)h_{ijk}^2\\
            &+ \sum_{i<j\leq m < k}(1 + \lambda_i\lambda_j + \lambda_j\lambda_k + \lambda_k\lambda_i + \lambda_i\frac{1+\lambda_k^2}{\lambda_i-\lambda_k} +\lambda_j\frac{1 + \lambda_k^2}{\lambda_j-\lambda_k})h_{ijk}^2\\
            &+\sum_{i\leq m < j<k}\lambda_i\left(\lambda_j + \lambda_k + \frac{1+\lambda_j^2}{\lambda_i-\lambda_j}+ \frac{1+\lambda_k^2}{\lambda_i-\lambda_k}\right)h_{ijk}^2\bigg]\\
            &+ \sum_{i=1}^m\frac{\lambda_i}{1 + \lambda_i^2}\pd_{ii}^2\Theta - m\sum_{a=1}^n\frac{\lambda_a}{1 + \lambda_a^2}(\pd_a\Theta)\pd_a b_m.
        \end{align*}

        As a function of matrices, $b_m$ is $C^2$ at $D^2u(x_0)$ with eigenvalues satisfying $\lambda = \lambda_1=\cdots=\lambda_m>\lambda_{m+1}$. We can approximate $D^2u(x_0)$ by matrices with distinct eigenvalues. The above expression for $\Delta_g b_m$ at $x_0$ holds and simplifies via part $(1)$ of Lemma \ref{y1} to
        \begin{align}
            m\Delta_g & b_m(x_0) = \sum_{k=1}^m(1 + \lambda^2)h_{kkk}^2 + \left(\sum_{i<k\leq m} + \sum_{k<i\leq m}\right)(3 + 3\lambda^2)h_{iik}^2\nonumber\\
            &\qquad+\sum_{k\leq m < i}\frac{2\lambda(1+\lambda\lambda_i)}{\lambda-\lambda_i}h_{iik}^2 + \sum_{i\leq m < k}\frac{3\lambda-\lambda_k + \lambda^2(\lambda + \lambda_k)}{\lambda-\lambda_k}h_{iik}^2\nonumber\\
            &\qquad+2\bigg[\sum_{i<j<k\leq m}(3 + 3\lambda^2)h_{ijk}^2\nonumber\\
            &\qquad+ \sum_{i<j\leq m < k}\left(1 + \frac{2\lambda}{\lambda-\lambda_k}+ \frac{\lambda^2(\lambda+\lambda_k)}{\lambda-\lambda_k}\right)h_{ijk}^2\nonumber\\
            &\qquad+\sum_{i\leq m < j<k}\lambda\left(\lambda_j + \lambda_k + \frac{1+\lambda_j^2}{\lambda-\lambda_j}+ \frac{1+\lambda_k^2}{\lambda-\lambda_k}\right)h_{ijk}^2\bigg]\nonumber\\
            &\qquad+ \sum_{i=1}^m\frac{\lambda_i}{1 + \lambda_i^2}\pd_{ii}^2\Theta - m\sum_{a=1}^n\frac{\lambda_a}{1 + \lambda_a^2}(\pd_a\Theta)\pd_a b_m\nonumber\\
            &\geq \sum_{k=1}^m\lambda^2h_{kkk}^2 + \left(\sum_{i<k\leq m} + \sum_{k< i\leq m}\right) 3\lambda^2 h_{iik}^2 + \sum_{k\leq m < i}\frac{2\lambda^2\lambda_i}{\lambda-\lambda_i}h_{iik}^2\label{lapbm}\\
            &\qquad + \sum_{i\leq m< k}\frac{\lambda^2(\lambda+\lambda_k)}{\lambda-\lambda_k}h_{iik}^2 + \sum_{i=1}^m\frac{\lambda_i}{1 + \lambda_i^2}\pd_{ii}^2\Theta - m\sum_{a=1}^n\frac{\lambda_a}{1 + \lambda_a^2}(\pd_a\Theta)\pd_a b_m.\nonumber
        \end{align}

        By \eqref{dgln} and the $C^1$ continuity of $b_m$ as a function of matrices at $D^2u(x_0)$, we have
        \begin{equation}\label{gradbm}
            |\nabla_g b_m|^2(x_0) = \frac{1}{m^2}\sum_{k=1}^n\lambda^2\left(\sum_{i=1}^m h_{iik}\right)^2\leq \frac{\lambda^2}{m}\sum_{k=1}^n\sum_{i=1}^m h_{iik}^2.
        \end{equation}

        Combining \eqref{lapbm} and \eqref{gradbm} we see that
        \begin{align}
            &m(\Delta_g b_m - \epsilon(n)|\nabla_g b_m|^2)\geq\nonumber\\
            & \lambda^2\left[\sum_{k=1}^m(1-\epsilon)h_{kkk}^2 + \left(\sum_{i<k\leq m}+\sum_{k<i\leq m}\right)(3-\epsilon)h_{iik}^2 + 2\sum_{k\leq m< i}\frac{\lambda_i}{\lambda-\lambda_i}h_{iik}^2\right]\label{Z1}\\
            &+ \lambda^2\left[\sum_{i\leq m < k}\left(\frac{\lambda+\lambda_k}{\lambda-\lambda_k}-\epsilon\right)h_{iik}^2\right]\label{Z2}\\
            &+ \sum_{i=1}^m\frac{\lambda_i}{1 + \lambda_i^2}\pd_{ii}^2\Theta - m\sum_{a=1}^n\frac{\lambda_a}{1 + \lambda_a^2}(\pd_a\Theta)\pd_a b_m\nonumber
        \end{align}
        where $\epsilon(n)$ will be fixed.

        \item[Step 2.] We will estimate each term in the above expression. For each fixed $k$ in the above expression, we set $t_i = h_{iik}$. For simplicity, we will denote
        \[
        H^k(x_0) = t_1(x_0) +\cdots + t_{n-1}(x_0) + t_n(x_0) = t'(x_0) + t_n(x_0)
        \]
        where $H^k$ denotes the $k$-th component of the mean curvature vector.

        \item[Step 2.1.] We first show that \eqref{Z1} can be bounded below by $-C(n,\nu_1)(1 + |Du(x_0)|^2)$. For each fixed $k\leq m$, we show that $[\;]_k \geq -C(n)\nu_1^2(1 + u^2_k(x_0))$. In the case that $\lambda_i\geq 0$ for all $i$, the proof follows directly. We will only consider the case $\lambda_{n-1}> 0 > \lambda_n$. For simplifying the notation, we assume $k=1$. From $t_n(x_0) = H^1(x_0) - t'(x_0)$ we observe:
        \begin{align*}
            \lambda^2[\;]_1 &= \lambda^2\left[(1-\epsilon)t_1^2 + \sum_{i=2}^m(3-\epsilon)t_i^2 + \sum_{i=m+1}^{n-1}\frac{2\lambda_i}{\lambda-\lambda_i}t_i^2 \right] + \frac{2\lambda_n}{\lambda-\lambda_n}t_n^2\\
            &= \lambda^2\left[(1-\epsilon)t_1^2 + \sum_{i=2}^m(3-\epsilon)t_i^2 + \sum_{i=m+1}^{n-1}\frac{2\lambda_i}{\lambda-\lambda_i}t_i^2 \right]\\
            &\quad + \lambda^2\frac{2\lambda_n}{\lambda-\lambda_n}[(H^1)^2 -2H^1t' + (t')^2]\\
            &\geq \lambda^2\left[(1-\epsilon)t_1^2 + \sum_{i=2}^m(3-\epsilon)t_i^2 + \sum_{i=m+1}^{n-1}\frac{2\lambda_i}{\lambda-\lambda_i}t_i^2\right]\\
            &\quad + \lambda^2\frac{2\lambda_n}{\lambda-\lambda_n}[(t')^2(1 + \eta)] + \lambda^2\frac{2\lambda_n}{\lambda-\lambda_n}[(H^1)^2(1 + \frac{1}{\eta})]
        \end{align*}
        where the last inequality follows from Young's inequality. Noting that $\frac{2\lambda_n}{\lambda-\lambda_n}\geq -\frac{2}{n}$ and using
        \[
        \lambda^2(H^1)^2= (\lambda g^{11}\pd_1\Theta)^2 = \left(\frac{\lambda_1}{1+\lambda_1^2}(\Theta_{x_1} + \Theta_uu_1 + \Theta_{u_1}\lambda_1)\right)^2\leq 3\nu_1^2(1 + u_1^2(x_0))
        \]
        we have
        \begin{align}
            \lambda^2[\;]_1&\geq  \lambda^2\left[(1-\epsilon)t_1^2 + \sum_{i=2}^m(3-\epsilon)t_i^2 + \sum_{i=m+1}^{n-1}\frac{2\lambda_i}{\lambda-\lambda_i}t_i^2\right]\nonumber\\
            &\qquad + \lambda^2\frac{2\lambda_n}{\lambda-\lambda_n}[(t')^2(1 + \eta)] - C(n)\nu_1^2(1 + u_1^2(x_0))(1 +\frac{1}{\eta})\nonumber\\
            &\geq \lambda^2\left[(1 - \epsilon)t_1^2 + \sum_{i=2}^m(3-\epsilon)t_i^2 + \sum_{i=m+1}^{n-1}\frac{2\lambda_i}{\lambda-\lambda_i}t_i^2 \right]\cdot\label{ts1}\\
            &\quad \left[1 + \frac{2(1 + \eta)\lambda_n}{\lambda - \lambda_n}\left(\frac{1}{1 -\epsilon} + \sum_{i=2}^m\frac{1}{3-\epsilon} + \sum_{i=m+1}^{n-1}\frac{\lambda-\lambda_i}{2\lambda_i}\right) \right]\label{ts2}\\
            &\quad- C(n)\nu_1^2( 1 + u_1^2(x_0))(1 +\frac{1}{\eta})\nonumber
        \end{align}
        where the last inequality follows from the Cauchy-Schwartz inequality. We have that \eqref{ts1} is positive choosing $\epsilon< 1$ and so it remains to find $\epsilon(n)$ such that \eqref{ts2} is positive. We observe
        \begin{align}
            &\left[1 + \frac{2(1 + \eta)\lambda_n}{\lambda - \lambda_n}\left(\frac{1}{1 -\epsilon} + \sum_{i=2}^m\frac{1}{3-\epsilon} + \sum_{i=m+1}^{n-1}\frac{\lambda-\lambda_i}{2\lambda_i}\right) \right]\nonumber\\
            &=\frac{2(1 + \eta)\lambda_n}{\lambda - \lambda_n}\left[\frac{\lambda - \lambda_n}{2(1 + \eta)\lambda_n} + \frac{1}{1 -\epsilon} +\frac{m-1}{3-\epsilon} + \sum_{i=m+1}^{n-1}\frac{\lambda-\lambda_i}{2\lambda_i}\right]\nonumber\\
            &= \frac{2(1 + \eta)\lambda_n}{\lambda - \lambda_n}\left[\frac{\lambda - \lambda_n}{2\lambda_n} - \frac{\lambda - \lambda_n}{2\lambda_n}\frac{\eta}{1+\eta}+ \frac{1}{1 -\epsilon} +\frac{m-1}{3-\epsilon} + \sum_{i=m+1}^{n-1}\frac{\lambda-\lambda_i}{2\lambda_i}\right]\nonumber\\
            &= \frac{2(1 + \eta)\lambda_n}{\lambda - \lambda_n}\left[ \frac{1}{1 -\epsilon} +\frac{m-1}{3-\epsilon} + \sum_{i=m+1}^{n}\frac{\lambda-\lambda_i}{2\lambda_i}\right] - \eta\nonumber\\
            &= \frac{2(1 + \eta)\lambda_n}{\lambda - \lambda_n}\left[ \frac{1}{1 -\epsilon} +\frac{m-1}{3-\epsilon} + \frac{\lambda}{2}\sum_{i=1}^{n}\frac{1}{\lambda_i} - \frac{n}{2}\right] - \eta\nonumber\\
            &= \frac{2(1 + \eta)\lambda_n}{\lambda - \lambda_n}\left[ \frac{1}{1 -\epsilon} +\frac{m-1}{3-\epsilon} + \frac{\lambda}{2}\frac{\sigma_{n-1}}{\sigma_n} - \frac{n}{2}\right] - \eta\nonumber\\
             &\geq \frac{2(1 + \eta)\lambda_n}{\lambda - \lambda_n}\left[ \frac{1}{1 -\epsilon} +\frac{m-1}{3-\epsilon} - \frac{n}{2}\right] - \eta \label{eta1}
        \end{align}
        where we used that $\lambda = \lambda_1 = \cdots=\lambda_m $ and Lemma \ref{y1}.
        
        Observe that finding an $\epsilon$ such that \eqref{eta1} is nonnegative, using $\frac{2\lambda_n}{\lambda-\lambda_n}\geq -\frac{2}{n}$, is equivalent to showing
        \[
        \frac{1}{1 -\epsilon} +\frac{m-1}{3-\epsilon} - \frac{n}{2} + \frac{\eta n}{4(1 + \eta)} \leq 0.
        \]
        Let
        \[
        a = \frac{n}{2} - \frac{\eta n}{4(1 + \eta)} = \frac{n}{4}\left(\frac{2+\eta}{1+\eta}\right),\qquad b = m-1.
        \]
        We get 
        \[
            \frac{1}{1 -\epsilon} +\frac{b}{3-\epsilon} - a \leq 0%
        \]
        which is equivalent to
        \[
        a\epsilon^2 - (4a - b- 1)\epsilon + (3a - b- 3) \geq 0. 
        \]
        
        The above function of $\epsilon$ has zeros at
        \begin{align*}
            \epsilon &= \frac{4a - b- 1 \pm \sqrt{(4a - b- 1)^2-4a(3a - b- 3)}}{2a}\\
            &= 2- \frac{2m}{n}\left(\frac{1+\eta}{2+\eta}\right) \pm\sqrt{\left(1 - \frac{2m}{n}\left(\frac{1+\eta}{2+\eta}\right)\right)^2 + \frac{8}{n}\left(\frac{1+\eta}{2+\eta}\right)}.
        \end{align*}
        Hence, we want to choose $\epsilon(n) >0$ such that
        \[
        \epsilon\leq 2- \frac{2m}{n}\left(\frac{1+\eta}{2+\eta}\right) - \sqrt{\left(1 - \frac{2m}{n}\left(\frac{1+\eta}{2+\eta}\right)\right)^2 + \frac{8}{n}\left(\frac{1+\eta}{2+\eta}\right)}.
        \]

        Let 
        \[
        f(m) = 2- \frac{2m}{n}\left(\frac{1+\eta}{2+\eta}\right) - \sqrt{\left(1 - \frac{2m}{n}\left(\frac{1+\eta}{2+\eta}\right)\right)^2 + \frac{8}{n}\left(\frac{1+\eta}{2+\eta}\right)}.
        \]
       We see that $f'(m) < 0$ and hence, it must be that
        \[
        f(n-1)\leq f(m) \leq f(1).
        \]
        By writing
        \[
        \frac{1+\eta}{2 +\eta} = \frac{1}{2}\left(1 + \frac{\eta}{2+\eta}\right)
        \]
        we have
        \[
            f(1) 
            = \frac{n-2}{n} - \frac{2}{n}\left(\frac{\eta}{2 + \eta}\right)
        \]
        and
        \begin{align*}
            f(n-1) &= 2 - \frac{n-1}{n}\left(1 + \frac{\eta}{2 + \eta}\right)\\
            &\hspace{.3in}- \sqrt{\left(1 - \frac{n-1}{n}\left(1 + \frac{\eta}{2 + \eta}\right) \right)^2 + \frac{4}{n}\left(1 + \frac{\eta}{2 + \eta}\right)}\\
            &= 1 -\frac{1}{n}\left(1 - (n-1)\frac{\eta}{2 + \eta} \right)\left[-1 +  \sqrt{1 + \frac{\frac{4}{n}\left(1 + \frac{\eta}{2+\eta}\right)}{\frac{1}{n^2}\left(1 - (n-1)\frac{\eta}{2 + \eta} \right)^2}}\;\;\right]\\
            &= 1 - \frac{\frac{\frac{4}{n}\left(1 + \frac{\eta}{2+\eta}\right)}{\frac{1}{n}\left(1 - (n-1)\frac{\eta}{2 + \eta} \right)}}{\sqrt{1 + \frac{\frac{4}{n}\left(1 + \frac{\eta}{2+\eta}\right)}{\frac{1}{n^2}\left(1 - (n-1)\frac{\eta}{2 + \eta} \right)^2}} + 1}\\
            &= 1 - \frac{\frac{4}{n}\left(1 + \frac{\eta}{2+\eta}\right)}{\sqrt{\frac{1}{n^2}\left(1 - (n-1)\frac{\eta}{2 + \eta} \right)^2 + \frac{4}{n}\left(1 + \frac{\eta}{2+\eta}\right)} + \frac{1}{n}\left(1 - (n-1)\frac{\eta}{2 + \eta} \right)}\\
            &= 1 - \frac{8( 1 + \eta)}{\sqrt{(2 - (n-2)\eta)^2 + 8n(1+\eta)(2 + \eta)} + 2 - (n-2)\eta}.
        \end{align*}
        In order to choose $0 < \epsilon(n) < f(n-1) \leq f(m)\leq f(1)$ we must first find an $\eta(n)$ such that
        \[
        \frac{8( 1 + \eta)}{\sqrt{(2 - (n-2)\eta)^2 + 8n(1+\eta)(2 + \eta)} + 2 - (n-2)\eta} < 1
        \]
        or rather
        \begin{align*}
            8(1+\eta) &< \sqrt{(2 - (n-2)\eta)^2 + 8n(1+\eta)(2 + \eta)} + 2 - (n-2)\eta\\
            (6+(n+6)\eta)^2 &< (2 - (n-2)\eta)^2 + 8n(1 + \eta)(2 + \eta)\\
            (8 + 8\eta)(4 + \eta(2n + 4)) &< 8n(1 + \eta)(2 + \eta)\\
            4 + \eta(2n + 4) &< n(2 + \eta)\\
            \eta & < \frac{2(n-2)}{n+4}.
        \end{align*}
        It suffices to choose $\eta < \frac{2}{7}$. We take $\eta = \frac{1}{4}$ and then choose $\epsilon(n)> 0$ such that
        \[
        \epsilon < 1-\frac{40}{\sqrt{n^2 + 340n + 100}+10-n} = 1 - \frac{40}{n\left(\sqrt{1 + \frac{340}{n}+\frac{100}{n^2}} - 1\right) + 10}.
        \]
        
        \item[Step 2.2] Now we show that \eqref{Z2} is nonnegative. For each $k$ such that $m < k < n$ we have $\lambda > \lambda_k > 0$ and hence $[\;]_k$ in \eqref{Z2} satisfies
        \begin{align*}
            [\;]_k &=\sum_{i=1}^m\left(\frac{\lambda+\lambda_k}{\lambda-\lambda_k} - \epsilon \right)t_i^2\\
            &\geq \sum_{i=1}^m(1-\epsilon)t_i^2 \geq 0
        \end{align*}
        for $\epsilon\leq 1$.
        \item[Step 2.3] For the $k=n$ term of \eqref{Z2}, we have
        \begin{align*}
            [\;]_n &= \sum_{i=1}^m\left(\frac{\lambda + \lambda_n}{\lambda - \lambda_n}-\epsilon \right)t_i^2\\
            &\geq \sum_{i=1}^m\left(\frac{n-2}{n}- \epsilon\right)t_i^2 \geq 0
        \end{align*}
        where the last line follows from Lemma \ref{y1} and choosing $\epsilon\leq \frac{n-2}{n}$.

        Altogether, we have shown that \eqref{Z1} and \eqref{Z2} are together bounded below by $-C(n)\nu_1^2(1+|Du(x_0)|^2)$ for $1\leq m \leq n-1$. That is,
        \begin{align}
            m(\Delta_g b_m - \epsilon(n)|\nabla_g b_m|^2)&\geq -C(n)\nu_1^2(1+|Du(x_0)|^2)\label{Z1Z2} \\
            &\quad + \sum_{i=1}^m\frac{\lambda_i}{1 + \lambda_i^2}\pd_{ii}^2\Theta - m\sum_{a=1}^n\frac{\lambda_a}{1 + \lambda_a^2}(\pd_a\Theta)\pd_a b_m.\label{Thet}
        \end{align}
        When $m=n$, we have that $\lambda_1 = \cdots=\lambda_n > 0$, and the same inequality holds.

        \item[Step 2.4] It remains to bound \eqref{Thet} from below.

        Note that for the phase $\Theta(x,u,Du)$, assuming the Hessian $D^2u$ is diagonalized at a point $x_0$, we get
\begin{align}
    \pd_i \Theta(x,u,Du) &= \Theta_{x_i} + \Theta_u u_i + \sum_k \Theta_{u_k}u_{ki} \label{dipsi}\\
    &\overset{x_0}{=} \Theta_{x_i} + \Theta_u u_i + \Theta_{u_i}\lambda_i. \nonumber
\end{align}
        Taking the $j$-th partial of \eqref{dipsi} we see that 
\begin{align*}
    \pd_{ij}\Theta(x,u,Du) &= \Theta_{x_ix_j} + \Theta_{x_i u}u_j + \sum_{r=1}^n \Theta_{x_iu_r}u_{rj}\nonumber\\
    & \qquad +\left(\Theta_{ux_j} + \Theta_{uu}u_j + \sum_{s=1}^n \Theta_{u u_s}u_{sj} \right)u_i + \Theta_u u_{ij}\nonumber\\
    & \qquad +\sum_{k=1}^n \left(\Theta_{u_kx_j} + \Theta_{u_ku}u_j + \sum_{t=1}^n \Theta_{u_ku_t}u_{tj}\right)u_{ki}+\sum_{k=1}^n \Theta_{u_k}u_{kij}\nonumber\\
    & \overset{x_0}{=} \Theta_{x_ix_j} + \Theta_{x_i u}u_j + \Theta_{x_iu_j}\lambda_j\label{dijpsi@p}\\
    & \qquad +\left(\Theta_{ux_j} + \Theta_{uu}u_j + \Theta_{u u_j}\lambda_j \right)u_i + \Theta_u \lambda_i\delta_{ij}\nonumber\\
    & \qquad +\left(\Theta_{u_ix_j} + \Theta_{u_iu}u_j + \Theta_{u_iu_j}\lambda_j\right)\lambda_i + \sum_{k=1}^n \Theta_{u_k}u_{kij}.\nonumber
\end{align*}
        
    Using the above, we see
        \begin{align}
            &\sum_{i=1}^m \frac{\lambda_i}{1 + \lambda_i^2}\pd_{ii}^2\Theta\nonumber\\
            &= \sum_{i=1}^m \frac{\lambda_i}{1 + \lambda_i^2}\Bigg[\Theta_{x_ix_i} + 2\Theta_{x_i u}u_i + 2\Theta_{x_iu_i}\lambda_i + 2\Theta_{uu_i}u_i\lambda_i + \Theta_{uu}u_i^2\\
            &\hspace{1.5in}+ \Theta_u\lambda_i + \Theta_{u_iu_i}\lambda_i^2 +\sum_{j=1}^n\Theta_{u_j}u_{jii}\Bigg]\nonumber\\
            &= \sum_{i=1}^m \frac{\lambda_i}{1 + \lambda_i^2}\Bigg[\Theta_{x_ix_i} + 2\Theta_{x_i u}u_i + 2\Theta_{x_iu_i}\lambda_i + 2\Theta_{uu_i}u_i\lambda_i + \Theta_{uu}u_i^2\\
            &\hspace{1.5in}+ \Theta_u\lambda_i + \Theta_{u_iu_i}\lambda_i^2\Bigg] +m\sum_{j=1}^n\Theta_{u_j}\pd_j b_m\nonumber.
        \end{align}

         Using \eqref{dipsi} we get
        \[
        m\sum_{a=1}^n\frac{\lambda_a}{1 + \lambda_a^2}(\pd_a\Theta)\pd_a b_m = m\sum_{a=1}^n\frac{\lambda_a}{1 + \lambda_a^2}(\Theta_{x_a} + \Theta_{u}u_a + \Theta_{u_a}\lambda_a)\pd_a b_m. 
        \]
        Hence, \eqref{Thet} becomes
        \begin{align}
            &\sum_{i=1}^m \frac{\lambda_i}{1 + \lambda_i^2}\pd_{ii}^2\Theta - m\sum_{a=1}^n\frac{\lambda_a}{1 + \lambda_a^2}(\pd_a\Theta)\pd_a b_m\nonumber \\
            &= \sum_{i=1}^m \frac{\lambda_i}{1 + \lambda_i^2}\Bigg[\Theta_{x_ix_i} + 2\Theta_{x_i u}u_i + 2\Theta_{x_iu_i}\lambda_i + 2\Theta_{uu_i}u_i\lambda_i + \Theta_{uu}u_i^2 + \Theta_u\lambda_i + \Theta_{u_iu_i}\lambda_i^2\Bigg]\label{Thetnu2}\\
            &\qquad + m\sum_{a=1}^n\frac{1}{1 + \lambda_a^2}(\Theta_{u_a}- \Theta_{x_a}\lambda_a - \Theta_{u}u_a\lambda_a)\pd_a b_m. \label{Thetnu1}
        \end{align}

        We bound \eqref{Thetnu2} below using the partial convexity of $\Theta(x,z,p)$ in $p$, by
        \begin{align}
            &\sum_{i=1}^m \frac{\lambda_i}{1 + \lambda_i^2}\Bigg[\Theta_{x_ix_i} + 2\Theta_{x_i u}u_i + 2\Theta_{x_iu_i}\lambda_i + 2\Theta_{uu_i}u_i\lambda_i + \Theta_{uu}u_i^2 + \Theta_u\lambda_i + \Theta_{u_iu_i}\lambda_i^2\Bigg]\nonumber\\
            &\geq \sum_{i=1}^m \frac{\lambda_i}{1 + \lambda_i^2}\Bigg[\Theta_{x_ix_i} + 2\Theta_{x_i u}u_i + 2\Theta_{x_iu_i}\lambda_i + 2\Theta_{uu_i}u_i\lambda_i + \Theta_{uu}u_i^2 + \Theta_u\lambda_i\Bigg]\nonumber\\
            &\geq -\sum_{i=1}^m \frac{\lambda_i}{1 + \lambda_i^2}\Bigg[|\Theta_{x_ix_i}| + 2|\Theta_{x_i u}u_i| + 2|\Theta_{x_iu_i}|\lambda_i + 2|\Theta_{uu_i}u_i|\lambda_i + |\Theta_{uu}|u_i^2 + |\Theta_u|\lambda_i\Bigg]\nonumber\\
            &\geq -C(n,\nu_1,\nu_2)( 1 + |Du(x_0)|^2)\label{Cnu2}
        \end{align}
        where the last line uses Young's inequality and that all the $\lambda_i> 0$ for $1\leq i\leq m$.

        Using Young's inequality we bound \eqref{Thetnu1} below by
        \begin{align}
            &m\sum_{a=1}^n\frac{1}{1 + \lambda_a^2}(\Theta_{u_a}- \Theta_{x_a}\lambda_a - \Theta_{u}u_a\lambda_a)\pd_a b_m\nonumber\\
            &\geq -m\sum_{a=1}^n \frac{1}{1+\lambda_a^2}(|\Theta_{u_a}|+ |\Theta_{x_a}\lambda_a| + |\Theta_{u}u_a\lambda_a|)|\pd_a b_m|\nonumber\\
            &\geq -\frac{\beta}{2}m^2|\nabla_g b_m|^2 - \frac{1}{2\beta}C(n,\nu_1)( 1 + |Du(x_0)|^2)\label{Cnu1}.
        \end{align}
        Combining \eqref{Z1Z2}, \eqref{Cnu2}, and \eqref{Cnu1}, we have
        \[
        m(\Delta_g b_m - (\epsilon(n) - \frac{\beta m}{2})|\nabla_g b_m|^2) \geq - C(n,\nu_1,\nu_2)(1 + \frac{1}{\beta})( 1 + |Du(x_0)|^2).
        \]
        Take $\beta = \epsilon(n)/m$ and let $c(n)= \epsilon(n)/2$ to get
        \[
        m(\Delta_g b_m - c(n)|\nabla_g b_m|^2) \geq - C(n,\nu_1,\nu_2)( 1 + |Du(x_0)|^2).
        \]
    \end{itemize}
\end{proof}

\begin{corollary}\label{scor}
    Let $u$ be a smooth solution to \eqref{s} in $B_1(0)\subset\re^n$ where $|\Theta|\geq (n-2)\frac{\pi}{2}$ and $n\geq 3$. Assuming the Hessian $D^2u$ is diagonalized at $x_0\in B_1(0)$, \eqref{ptJI} holds with $C= C(n,s_2)(1+|Du(x_0)|^2)$.
\end{corollary}
\begin{proof}
    Let $x_0\in B_1$. As $\Theta(x,u,Du) = s_1 + s_2(x\cdot Du - 2u)$, we get that
    \[
    \begin{matrix}
        \Theta_{x_i}= s_2u_i & \Theta_{x_ix_j}= 0 & \Theta_{x_iu}= 0&\Theta_{x_iu_j}=s_2\delta_{ij}\\
        \Theta_u = -2s_2 & \Theta_{ux_j} = 0& \Theta_{uu} = 0 &\Theta_{uu_j}=0\\
        \Theta_{u_i} = s_2x_i & \Theta_{u_ix_j}= s_2\delta_{ij} & \Theta_{u_iu}= 0 & \Theta_{u_iu_j}= 0.
    \end{matrix}
    \]
    Hence \eqref{Thetnu2} becomes zero and \eqref{Thetnu1} becomes
    \begin{equation*}
        m\sum_{k=1}^n\frac{s_2}{1+ \lambda_k^2}\left(x_k + u_k\lambda_k\right)\pd_k b_m.
    \end{equation*}
    Applying Young's inequality and simplifying, we get
    \begin{align*}
        m(\Delta_g b_m - c(n)|\nabla_g b_m|^2_g) &\geq -\frac{n^2s_2^2}{2\epsilon(n)}\left(|x_0|^2 + |Du(x_0)|^2\right) - C(n,\nu_1)(1 + |Du(x_0)|^2) \\
        &\geq 
        -C(n,s_2)(1+|Du(x_0)|^2).
    \end{align*}
\end{proof}

\begin{corollary}\label{tcor}
    Let $u$ be a smooth solution to \eqref{tran} in $B_1(0)\subset\re^n$ with $|\Theta|\geq (n-2)\frac{\pi}{2}$ and $n\geq 3$. Assuming the Hessian $D^2u$ is diagonalized at $x_0\in B_1(0)$, \eqref{ptJI} holds with $C= C(n,t_2,t_3)$.
\end{corollary}
\begin{proof}
    As $\Theta(x,u,Du) = t_1 + t_2\cdot x + t_3\cdot Du$, we get 
    \[
    \Theta_{x_i} = t_{2,i} \quad\text{ and }\quad \Theta_{u_i} = t_{3,i}
    \]
    where all the remaining derivatives are zero.
    Hence \eqref{Thetnu2} is zero and \eqref{Thetnu1} becomes
    \begin{equation*}
        m\sum_{k=1}^n\frac{1}{1+ \lambda_k^2}\left(t_{3,k} - t_{2,k}\lambda_k\right)\pd_kb_m.
    \end{equation*}
    Applying Young's inequality and simplifying, we get
    \begin{align*}
        m(\Delta_g b_m - c(n)|\nabla_g b_m|^2_g )&\geq -\frac{n^2}{2\epsilon(n)}\left(|t_2|^2 + |t_3|^2\right) - C(n,\nu_1)(1 + |Du(x_0)|^2)\\
        &\geq -C(n,t_2,t_3)(1 + |Du(x_0)|^2).
    \end{align*}
\end{proof}

\begin{corollary}\label{rcor}
    Let $u$ be a smooth solution to \eqref{rotator} in $B_1(0)\subset\re^n$ with $|\Theta|\geq (n-2)\frac{\pi}{2}$ and $n\geq 3$. Assuming the Hessian $D^2u$ is diagonalized at $x_0\in B_1(0)$, \eqref{ptJI} holds with $C= C(n,r_2)(1 + |Du(x_0)|^2)$.
\end{corollary}
\begin{proof}
    Let $x_0\in B_1$. As $\Theta(x,u,Du) = r_1 + \frac{r_2}{2}(|x|^2 + |Du|^2)$, we get 
    \[
    \begin{matrix}
        \Theta_{x_i}= r_2x_i & \Theta_{x_ix_j}= r_2\delta_{ij} &\Theta_{x_iu_j}=0\\
        \Theta_{u_i} = r_2u_i & \Theta_{u_ix_j}= 0  & \Theta_{u_iu_j}= r_2\delta_{ij}.
    \end{matrix}
    \]
    Then \eqref{Thetnu2} and \eqref{Thetnu1} are bounded below by
    \begin{align*}
        \sum_{a=1}^m\frac{\lambda_a}{1+\lambda_a^2}\bigg[r_2 +&r_2\lambda_a^2\bigg]+ m\sum_{k=1}^n\frac{r_2}{1+ \lambda_k^2}\left(u_k - x_k\lambda_k\right)\pd_kb_m\\
        &\geq m\sum_{k=1}^n\frac{r_2}{1+ \lambda_k^2}\left(u_k - x_k\lambda_k\right)\pd_kb_m
    \end{align*}
    since $r_2\geq 0$ and $\lambda_a\geq0$ for all $1\leq a\leq m$. Thus, using Young's inequality and simplifying, we get
    \begin{align*}
        m(\Delta_g b_m - c(n)|\nabla_g b_m|^2_g) &\geq -\frac{n^2r_2^2}{2\epsilon(n)}\left(|x_0|^2 + |Du(x_0)|^2\right)- C(n,\nu_1)(1 + |Du(x_0)|^2) \\
        &\geq -C(n,r_2)(1 + |Du(x_0)|^2).
    \end{align*}
\end{proof}

We now extend the above Jacobi inequalities to the following integral form.
\begin{proposition}\label{IntJI}
    Let $u$ be a smooth solution to \eqref{slag} in $\re^n$ with $n\geq 3$ and $|\Theta|\geq (n-2)\frac{\pi}{2}$ where $\Theta(x,z,p)$ is partially convex in the $p$ variable. Let
    \begin{equation*}
        b=b_1 = \log\sqrt{1 + \lambda_{\max{}}^2}
    \end{equation*}
    where $\lambda_{\max{}}$ is the largest eigenvalue of $D^2u$, i.e. $\lambda_{\max{}} = \lambda_1 \geq \lambda_2\geq\cdots\geq\lambda_n$. Then, for all non-negative $\phi\in C_0^\infty(B_R)$, $b$ satisfies the integral Jacobi inequality
    \begin{equation*}
        \int_{B_R}-\langle\nabla_g \phi, \nabla_g b\rangle_g\;dv_g \geq c(n) \int_{B_R} \phi|\nabla_g b|^2 dv_g - \int_{B_R}C(n,\nu_1,\nu_2)(1 + |Du(x)|^2)\phi dv_g.
    \end{equation*}
    From this, it follows that, for $r < R$, 
    \begin{equation*}
        \int_{B_r}|\nabla_g b|^2dv_g \leq \left(\frac{C(n)}{R-r} + C(n,\nu_1,\nu_2)(1 + ||Du||^2_{L^\infty(B_R)}) \right)\int_{B_R}\;dv_g.
    \end{equation*}
\end{proposition}
\begin{remark}
    This holds via the same proof for solutions to \eqref{s},\eqref{tran},\eqref{rotator} each with their respective $C$.
\end{remark}

\begin{proof} The proof follows from \cite{AB}. We present it here for the sake of completion.
    If $b$ is smooth, then the integral Jacobi inequality follows immediately from the pointwise Jacobi inequality \eqref{ptJI}. However, $b$ is only Lipschitz in terms of entries of $D^2u$ and since $u$ is smooth in $x$, we have $b$ is only Lipschitz in $x$. We want to show that \eqref{ptJI} holds in the viscosity sense. 

    Let $x_0\in B_R(0)$ and let $P$ be a quadratic polynomial such that $P(x_0) = b(x_0)$ and elsewhere
    \[
    P(x)\geq b(x).
    \]
    If $x_0$ is a smooth point of $b$, then via \eqref{ptJI} with $m=1$ we get
    \[
    \Delta_g P \geq c(n)|\nabla_g P|^2 - C(n,\nu_1,\nu_2)(1 + |Du(x_0)|^2)
    \]
    otherwise, $\lambda_1$ is not distinct at $x_0$. 
    If it is not distinct at $x_0$ and we have $\lambda_1 = \cdots=\lambda_k > \lambda_{k+1}$, then the function $b_k$ will be smooth at $x_0$ from Lemma \ref{ptJ}. Observe that we have
    \[
    P(x) \geq b(x) \geq b_k(x) \text{ with equality holding at $x_0$.}
    \]
    Applying \eqref{ptJI} with $b_k$, we see that at $x_0$ we still have
    \[
    \Delta_g P \geq c(n)|\nabla_g P|^2 - C(n,\nu_1,\nu_2)(1 + |Du(x_0)|^2).
    \]
    This shows that \eqref{ptJI} holds in the viscosity sense. By applying the result of Ishii \cite[Theorem 1]{Ish}, the viscosity subsolution $b$ to \eqref{ptJI} is also a distribution subsolution. We can then integrate by parts against $\phi \in C^\infty_0(B_R)$ to get 
    \begin{align*}
        -\int_{B_R}\langle\nabla_g\phi,\nabla_gb\rangle_g dv_g &= \int_{B_R}\phi \Delta_g b \; dv_g\\
        &\geq c(n)\int_{B_R}\phi|\nabla_g b|^2\;dv_g - \int_{B_R}C(n,\nu_1,\nu_2)(1 + |Du(x)|^2)\;dv_g.
    \end{align*}
    Let $0< r< R$. Choose $0\leq \phi \leq 1$ with $\phi=1$ on $B_r$ and $|D\phi|\leq \frac{2}{R-r}$ in $B_R$, and use the above inequality to get
    \begin{align*}
        \int_{B_r}|\nabla_g b|^2dv_g&\leq \int_{B_R}\phi^2|\nabla_g b|^2dv_g\\
        &\leq \frac{4}{c(n)^2}\int_{B_R}|\nabla_g\phi|^2dv_g + \frac{2}{c(n)}\int_{B_R} \phi^2 C(n,\nu_1,\nu_2)(1 + |Du(x_0)|^2)\; dv_g\\
        &\leq \left(\frac{C(n)}{R-r} + C(n,\nu_1,\nu_2)(1 + ||Du||^2_{L^\infty(B_R)}) \right)\int_{B_R}\;dv_g.
    \end{align*}
\end{proof}

Observe that the results proved in this section hold good when $\Theta$ is critical and supercritical.
However, the remainder of the proof in this paper does not hold at the critical value and requires $\Theta$
to be supercritical. 

\section{Lewy-Yuan rotation}\label{LYr}

We use a modified version of the Lewy-Yuan rotation adapted to the supercritical phase case \cite[Section 4]{CW2}. Let $\delta \in (0,\pi/2)$ and choose $\alpha < \delta$. We rotate our gradient graph downwards by an angle of  $e^{-i\alpha}$:
\begin{equation}\label{lyrot}
\begin{cases}
     \bar{x} = \cos(\alpha) x + \sin(\alpha) Du(x)\\
    \bar{y} = D\bar{u} = -\sin(\alpha) x + \cos(\alpha) Du(x).
\end{cases}
\end{equation} 

Upon rotating we have 
\[
-\cot(\delta - \alpha)I_n \leq D^2\bar{u}\leq \cot(\alpha)I_n.
\]
As $\alpha < \delta$ we have $\cot(\delta) < \cot(\alpha)$. So there exists some $\eta$ such that $\cot(\delta) = \cot(\alpha) - \eta$. We then have 
\[
u + \cot(\delta)\frac{|x|^2}{2} = u + (\cot(\alpha) - \eta)\frac{|x|^2}{2} 
\]
is convex. By \cite[Prop 4.1]{CW2} under the rotation the new coordinates exist in a ball of radius
\[
\bar{R} \geq \sin(\alpha)\eta R.
\]
We take $\alpha = \delta/2$. We have 
\begin{equation}\label{dbaru}
-\cot(\delta/2)I_n \leq D^2\bar{u}\leq \cot(\delta/2)I_n,
\end{equation}
and hence,
\begin{equation}\label{barg}
   d\bar{x}^2 \leq g(\bar{x}) \leq (1 + \cot^2(\delta/2))^\frac{n}{2} d\bar{x}^2 = \csc^n(\delta/2)d\bar{x}^2
\end{equation}
where our coordinates exist in a radius of
\[
\bar{R}\geq \frac{1}{2\cos(\delta/2)}R.
\]
We define 
\[
\bar{\Omega}_r = \bar{x}(B_r(0))
\]
and see that
\begin{equation}\label{rhodr}
 |\bar{x}|\leq \cos(\delta/2)r + \sin(\delta/2)||Du||_{L^\infty(B_r)}\leq r + ||Du||_{L^\infty(B_r)}=\rho(r).   
\end{equation} 

We also observe that
\[
\text{dist}(\bar{\Omega}_r,\pd\bar{\Omega}_{R}) \geq \frac{R-r}{2\cos(\delta/2)}\geq \frac{R-r}{\sqrt{2}}.
\]

\begin{proposition}\label{Iso}
    Let $u$ be a smooth semi-convex function with $D^2u\geq -\cot(\delta)I_n$ on $B_{R'}(0)\subset \re^n$. Let $f$ be a smooth positive function on the Lagrangian surface $X = (x,Du(x))$. Let $0<r< R < R'$ be such that $R-r > 2\epsilon$. Then
    \[
    \left[\int_{B_r}|(f-\tilde{f})^+|^\frac{n}{n-1}\;dv_g \right]^\frac{n-1}{n} \leq C(n)\left(\frac{\csc^2(\delta/2)\rho^2}{r\epsilon}\right)^{(n-1)}\int_{B_R}|\nabla_g(f-\tilde{f})^+|\;dv_g
    \]
    where $\rho= \rho(R')$ is as defined in \eqref{rhodr}, and
    \[
    \tilde{f} = \frac{2}{|B_r|}\int_{B_R(0)} f\;dx.
    \]
\end{proposition}

We first state \cite[Lemma 4.1]{BW}
, which is a generalization of \cite[Lemma 3.2]{CWY}.
\begin{lemma}\label{IsoLem}
    Let $\Omega_1\subset\Omega_2\subset B_{\rho}\subset\re^n$ and $\epsilon>0$. Suppose that dist$(\Omega_1,\pd\Omega_2)\geq 2\epsilon$; $A$ and $A^c$ are disjoint measurable sets such that $A\cup A^c = \Omega_2$. Then
    \[
        \min\{|A\cap\Omega_1|,|A^c\cap \Omega_1|\}\leq C(n)\frac{\rho^n}{\epsilon^n}|\pd A\cap \pd A^c|^\frac{n}{n-1}.
    \]
\end{lemma}
\begin{proof}[Poof of Proposition \ref{Iso}] For the sake of completion we add a proof tailored to the supercritical phase $|\Theta|\geq (n-2)\frac{\pi}{2}+\delta$.
Let $M=||f||_{L^\infty(B_r)}$. If $M\leq \ti{f}$, then $(f-\ti{f})^+=0$ on $B_r$, and hence, the left hand side is zero, from which the result follows immediately. We assume $\ti{f}< M$. By the Morse-Sard Lemma \cite[Lemma 13.15]{maggi}, \cite{sard}, $\{x | f(x) = t\}$ is $C^1$ for almost all $t \in (\ti{f},M)$. We first show that for such $t$,
\begin{equation}\label{fiso}
   |\{x | f(x)> t\}\cap B_r|_g\leq C(n)\frac{\csc^{2n}(\delta/2)\rho^{2n}}{r^n\epsilon^{n}}|\{x | f(x) = t\}\cap B_R|_g^\frac{n}{n-1}. 
\end{equation}
Note $|\cdot|_g$ is the metric with respect to $g$, and $|\cdot|$ is the Euclidean metric.

Let $t>\ti{f}$. It must be that
\[
    \frac{|B_r|}{2} > |\{x | f(x) > t\}\cap B_r|
\]
since otherwise
\[
M = \frac{2}{|B_r|}\int_0^M\frac{|B_r|}{2}dt\leq \frac{2}{|B_r|}\int_0^M|\{x|f(x)>t\}\cap B_r|dt \leq \frac{2}{|B_r|}\int_{B_R}fdx=\ti{f} < M.
\]
From this, it follows
\begin{equation}\label{Atrot}
    |\{x | f(x) \leq t\}\cap B_r| > \frac{|B_r|}{2}.
\end{equation}
If $|A_t\cap \bar{\Omega}_r|\leq|A_t^c\cap \bar{\Omega}_r|$, then via \eqref{barg} and Lemma \ref{IsoLem},
\begin{align*}
    |A_t\cap\bar{\Omega}_r|_{g(\bar{x})}&\leq \csc^n(\delta/2)|A_t\cap\bar{\Omega}_r|\\
    &\leq C(n)\csc^n(\delta/2)\frac{\rho^n}{\epsilon^n}|\pd A_t\cap \pd A_t^c|_{g(\bar{x})}^\frac{n}{n-1}.
\end{align*}

On the other hand, if $|A_t\cap \bar{\Omega}_r|>|A_t^c\cap\bar{\Omega}_r|$, from \eqref{Atrot}, we have 
\[
|B_r| < 2\csc^n(\delta/2)|A_t^c\cap\bar{\Omega}_r|
\]
and so 
\[
|A_t\cap\bar{\Omega}_r|\leq \frac{\rho^n}{r^n}|B_r|\leq 2\csc^n(\delta/2)\frac{\rho^n}{r^n}|A_t^c\cap\bar{\Omega}_r|.
\]
Therefore, via Lemma \ref{IsoLem},
\[
|A_t\cap\bar{\Omega}_r|_{g(\bar{x})} \leq 2\csc^n(\delta/2)\frac{\rho^n}{r^n}|A_t^c\cap\bar{\Omega}_r|\leq C(n)\csc^{2n}(\delta/2)\frac{\rho^{2n}}{r^n\epsilon^n}|\pd A_t\cap \pd A_t^c|_{g(\bar{x})}^\frac{n}{n-1}.
\]
In either case, we have
\begin{equation*}
    |A_t\cap\bar{\Omega}_r|_{g(\bar{x})}\leq C(n)\csc^{2n}(\delta/2)\frac{\rho^{2n}}{r^n\epsilon^n}|\pd A_t\cap \pd A_t^c|_{g(\bar{x})}^\frac{n}{n-1}
\end{equation*}
which in our original coordinates is \eqref{fiso}.

Thus, we get
\begin{align*}
    \bigg[\int_{B_r}&|(f-\ti{f})^+|^\frac{n}{n-1}dv_g\bigg]^\frac{n-1}{n}\\
    &=\left[\int_0^{M-\ti{f}}|\{x | f(x) - \ti{f}> t\}\cap B_r|_g dt^\frac{n}{n-1}\right]^\frac{n-1}{n} \\
    &\leq \int_0^{M-\ti{f}}|\{ x | f(x) - \ti{f}> t\}\cap B_r|^\frac{n-1}{n}_g dt\\ 
    &\leq C(n)\left(\frac{\csc^2(\delta/2)\rho^2}{r\epsilon}\right)^{(n-1)}\int_{\ti{f}}^M|\{x| f(x) = t\}\cap B_R|_g dt \text{ via \eqref{fiso}}\\
    &\leq C(n)\left(\frac{\csc^2(\delta/2)\rho^2}{r\epsilon}\right)^{(n-1)}\int_{B_R}|\nabla_g(f - \ti{f})^+|dv_g 
\end{align*}
The first equality is via the Layer cake \cite[Ex 1.13]{maggi}; the line after that is via the H-L-P inequality \cite[(5.3.3)]{LinYang}; and the last inequality comes from the co-area formula \cite[Thm 4.2.1]{LinYang}. This completes the proof.
\end{proof}

\section{Gradient estimate}

By Lemma \ref{y1},  $D^2u \geq -\cot(\delta)I_n$. Set $\tilde{u} = u + \cot(\delta)\frac{|x|^2}{2}$. From this we have $D^2\tilde{u}\geq 0$, so $\tilde{u}$ is convex. As $\tilde{u}$ is convex, we have
\[
||D\tilde{u}||_{L^\infty(B_r(0))}\leq \osc_{B_{r+1}(0)}\tilde{u}.
\]
That is,
\begin{equation*}
    ||Du + \cot(\delta)x||_{L^\infty(B_r(0))}\leq \osc_{B_{r+1}(0)}(u) + \cot(\delta)\frac{(r+1)^2}{2}.
\end{equation*}
It follows that
\begin{align}
    ||Du||_{L^\infty(B_r(0))}&= ||Du + \cot(\delta)x-\cot(\delta)x||_{L^\infty(B_r(0))}\nonumber\\
    &\leq ||Du + \cot(\delta)x||_{L^\infty(B_r(0))} + || \cot(\delta)x||_{L^\infty(B_r(0))}\nonumber\\
    &\leq \osc_{B_{r+1}(0)}(u) + \cot(\delta)\frac{(r+1)^2}{2} + \cot(\delta)r\nonumber\\
    &= \osc_{B_{r+1}(0)}(u) +\frac{\cot(\delta)}{2}(r^2 + 4r + 1). \label{gradest}
\end{align}
\section{Proof of the main results}
We prove Theorem \ref{main1} from which Theorem \ref{main0} follows.

\begin{proof} Let $n\geq 3$. To simplify the remaining proof's notation, we take $R = 2n+2$ where $u$ is a solution on $B_{2n+2}\subset\re^n$. Then by scaling $v(x) = \frac{u(\frac{R}{2n+2}x)}{(\frac{R}{2n+2})^2}$, we get the estimate in Theorem \ref{main1}. We denote $C=C(n,\nu_1,\nu_2)(1 + ||Du||^2_{L^\infty(B_{2n+1})})$ the positive constant from Lemma \ref{ptJ}.
\begin{enumerate}
    \item[Step 1.] We use the rotated Lagrangian graph $X=(\bar{x},D\bar{u}(\bar{x}))$ via the Lewy-Yuan rotation, illustrated in Section \ref{LYr}. Consider $b = \sqrt{1 + \lambda_{\max{}}^2}$ on the manifold $X=(x,Du(x))$. In the rotated coordinates, $b(\bar{x})$ weakly satisfies
    \begin{align}
    \bigg(g^{ij}(\bar{x})\frac{\pd^2}{\pd\bar{x}_i\pd\bar{x}_j}-& g^{jp}(\bar{x})\frac{\pd \Theta(x(\bar{x}),u(x(\bar{x})), \sin(\delta/2)\bar{x} + \cos(\delta/2) D\bar{u}(\bar{x}))}{\pd\bar{x}_q}\frac{\pd^2 \bar{u}(\bar{x})}{\pd \bar{x}_q\pd\bar{x}_p} \frac{\pd}{\pd \bar{x}_j}\bigg) b(\bar{x})\nonumber\\
    &= \Delta_{g(\bar{x})}b(\bar{x}) \geq - C. \label{subsol}
\end{align}
The nondivergence and divergence elliptic operator are both uniformly elliptic due to \eqref{dbaru}.

From \eqref{lyrot}, we have
\[
\begin{cases}
    x(\bar{x}) = \cos(\delta/2)\bar{x} - \sin(\delta/2)D\bar{u}(\bar{x})\\
    Du(x(\bar{x})) = \sin(\delta/2)\bar{x} + \cos(\delta/2)D\bar{u}(\bar{x})
\end{cases}
\]
from which it follows that
\begin{align}
    &\frac{\pd \Theta(x(\bar{x}),u(x(\bar{x})), \sin(\delta/2)\bar{x} + \cos(\delta/2)D\bar{u}(\bar{x}))}{\pd\bar{x}_q}\frac{\pd^2 \bar{u}(\bar{x})}{\pd \bar{x}_q\pd\bar{x}_p}\nonumber\\
    &= \sum_{j=1}^n\Theta_{x_j}\frac{\pd x_j}{\pd \bar{x}_q}\bar{\lambda}_q + \Theta_u\sum_{j=1}^n u_j\frac{\pd x_j}{\pd\bar{x}_q}\bar{\lambda}_q + \sum_{j=1}^n\Theta_{u_j}\frac{\pd}{\pd\bar{x}_q}\bigg(\sin(\delta/2)\bar{x}_j + \cos(\delta/2)\bar{u}_j\bigg)\bar{\lambda}_q\nonumber\\
    &= \Theta_{x_q}\bigg(\cos(\delta/2) - \sin(\delta/2)\bar{\lambda}_q\bigg)\bar{\lambda}_q \nonumber\\
    &\qquad + \Theta_{u}u_q\bigg(\cos(\delta/2) - \sin(\delta/2)\bar{\lambda}_q \bigg)\bar{\lambda}_q\nonumber\\
    &\qquad +  \Theta_{u_q}\bigg(\sin(\delta/2) + \cos(\delta/2)\bar{\lambda}_q \bigg)\bar{\lambda}_q\nonumber\\
    &\leq |\Theta_{x_q}|\bigg(\cos(\delta/2) + \sin(\delta/2)\cot(\delta/2)\bigg)\cot(\delta/2) \nonumber\\
    &\qquad + |\Theta_{u}u_q|\bigg(\cos(\delta/2) + \sin(\delta/2)\cot(\delta/2) \bigg)\cot(\delta/2)\nonumber\\
    &\qquad +  |\Theta_{u_q}|\bigg(\sin(\delta/2) + \cos(\delta/2)\cot(\delta/2) \bigg)\cot(\delta/2)\nonumber\\
    &= 2\cos(\delta/2)\cot(\delta/2)\bigg(|\Theta_{x_q}| + |\Theta_{u}u_q| \bigg) + \frac{\cot(\delta/2)}{\sin(\delta/2)}|\Theta_{u_q}|\nonumber\\
    &\leq 4\csc^2(\delta/2)\nu_1(1 + |u_q|).\label{d1bound}
\end{align}

We denote
\[
\tilde{b} = \frac{2}{|B_1(0)|}\int_{B_{2n}(0)}b(x)\;dx.
\]
Via the local mean value property of nonhomogeneous subsolutions \cite[Theorem 9.20]{GT}, (see \cite[Appendix, Theorem 6.1]{BW}), 
we get the following, from \eqref{subsol} and \eqref{d1bound}:
\begin{align*}
    (b - \tilde{b})^+(0) &= (b - \tilde{b})^+(\bar{0})\\
    &\leq C(n)\left[\ti{C}^{\;n-1}\left(\int_{B_{1/2}(\bar{0})}|(b - \tilde{b})^+(\bar{x})|^\frac{n}{n-1}d\bar{x} \right)^\frac{n-1}{n} + C\left(\int_{B_{1/2}(\bar{0})}d\bar{x}\right)^\frac{1}{n}\right]\\
    &\leq C(n)\left[\ti{C}^{\;n-1}\left(\int_{B_{1/2}(\bar{0})}|(b - \tilde{b})^+(\bar{x})|^\frac{n}{n-1}dv_{g(\bar{x})} \right)^\frac{n-1}{n} + C\left(\int_{B_{1/2}(\bar{0})}dv_{g(\bar{x})}\right)^\frac{1}{n}\right]\\
    &\leq C(n)\left[\ti{C}^{\;n-1}\left(\int_{B_1(0)}|(b - \tilde{b})^+(x)|^\frac{n}{n-1}dv_{g(x)} \right)^\frac{n-1}{n} + C\left(\int_{B_1(0)}dv_g\right)^\frac{1}{n}\right]
\end{align*}
where $\ti{C}=(1 + \nu_1\csc^2(\delta/2)(1 + ||Du||_{L^\infty(B_{2n+1})}))$ and $C=C(n,\nu_1,\nu_2)(1 + ||Du||^2_{L^\infty(B_{2n+1})})$ is the positive constant from Lemma \ref{ptJ}. 

The above mean value inequality can also be derived using the De Giorgi-Moser iteration \cite[Theorem 8.16]{GT}.

\item[Step 2.] By Proposition \ref{Iso} with $\rho = \rho(2n+1)$ and $\epsilon = \frac{5}{4}$, and Lemma \ref{IntJI}, approximating $b$ by a smooth function, we have
\begin{align}
    b(0)&\leq C(n)\tilde{C}^{\;n-1}\csc(\delta/2)^{2n-2}\rho^{2n-2}\int_{B_{2n}}|\nabla_g(b-\tilde{b})^+|\;dv_g \nonumber \\
   &\qquad + CC(n)\left(\int_{B_{2n}}dv_g \right)^\frac{1}{n} + C(n) \int_{B_{2n}}b(x)dx \nonumber\\
    &\leq C(n)\ti{C}^{\;n-1}\csc(\delta/2)^{2n-2}\rho^{2n-2}\left(\int_{B_{2n}}|\nabla_g b|^2dv_g\right)^\frac{1}{2}\left(\int_{B_{2n}}dv_g\right)^\frac{1}{2} \nonumber\\
    &\qquad+CC(n)\left(\int_{B_{2n}}dv_g\right)^\frac{1}{n}+ C(n)\int_{B_{2n}} dv_g\nonumber\\
    &\leq C(n)(1 + \ti{C}^{\;n-1}(1 + C)^\frac{1}{2})\csc(\delta/2)^{2n-2}\rho^{2n-2}\int_{B_{2n+1}}dv_g+ CC(n)\left(\int_{B_{2n+1}}dv_g\right)^\frac{1}{n}. \label{b0}
\end{align}

\item[Step 3.] We bound the volume element using the rotated coordinates. From \eqref{barg}, and $\bar{\Omega}_{2n+1}=\bar{x}(B_{2n+1}(0))$, we get
\[
\int_{B_{2n+1}}dv_{g(x)} = \int_{\bar{\Omega}_{2n+1}}dv_{g(\bar{x})}\leq \csc^n(\delta/2)\int_{\bar{\Omega}_{2n+1}}d\bar{x}\leq C(n)\csc^n(\delta/2)\rho^n.
\]
Hence, from \eqref{b0}, we have
\begin{align}
    b(0) &\leq C(n)\csc^{3n-2}(\delta/2)(1 + \tilde{C}^{n-1}( 1 + C)^\frac{1}{2})\rho^{3n-2} + CC(n)\csc(\delta/2)\rho\nonumber\\
    &\leq C(n)\csc^{3n-2}(\delta/2)(1 + \tilde{C}^{n-1}( 1 + C)^\frac{1}{2} + C)\rho^{3n-2}\label{b1}.
\end{align}

By plugging in \eqref{rhodr}, $\tilde{C}$, $C$, and using
\[
(a + b)^p \leq 2^p(a^p + b^p),\quad \text{ for } a,b\geq0,p>0,
\]
$\csc(\delta/2)\geq 1$, as well as Young's inequality, we have
\begin{align}
     C(n)&\csc^{3n-2}(\delta/2)(1 + \tilde{C}^{n-1}( 1 + C)^\frac{1}{2} + C)\rho^{3n-2}\nonumber\\
     &\leq C(n,\nu_1,\nu_2)\csc^{5n-4}(\delta/2)(1 +||Du||^{4n-2}_{L^\infty(B_{2n+1})} ).\nonumber
\end{align}
Using the gradient estimate \eqref{gradest} from above, this reduces to
\begin{align}
    C(n,\nu_1,\nu_2)&\csc^{5n-4}(\delta/2)(1 +||Du||^{4n-2}_{L^\infty(B_{2n+1})} )\nonumber\\
    & \leq C(n,\nu_1,\nu_2)\csc^{5n-4}(\delta/2)(1 + (\osc_{B_{2n+2}}(u))^{4n-2} + \cot^{4n-2}(\delta))\nonumber\\
    &\leq C(n,\nu_1,\nu_2)\csc^{9n-6}(\delta/2)(1 + (\osc_{B_{2n+2}}(u))^{4n-2}). \label{b2} 
\end{align}

By combining \eqref{b1} and \eqref{b2} and exponentiating, we get
\[
|D^2u(0)|\leq \exp\big[C_1\csc^{9n-6}(\delta/2)\big]\exp\bigg[C_2\csc^{9n-6}(\delta/2)\big(\osc_{B_{2n+2}}(u)\big)^{4n-2}\bigg]
\]
where $C_1$ and $C_2$ are positive constants depending on $\nu_1,\nu_2$ and $n$.
\end{enumerate}  
Lastly, we deal with the case $n=2$. We fix $\arctan\lambda_3 = \frac{\pi}{2} -\frac{\delta}{2}$ and add $(\pi/2 - \delta/2)$ to both sides of the two dimensional supercritical equation \eqref{slag}:
\[
\arctan\lambda_1 + \arctan\lambda_2 + \arctan\lambda_3 = \Theta(x,u,Du) + \frac{\pi}{2} - \frac{\delta}{2}\geq \frac{\pi}{2} + \frac{\delta}{2}.
\]
This is then the three dimensional supercritical equation \eqref{slag} for which the Hessian estimates hold by the above proof.

\end{proof}

\bibliographystyle{amsalpha}
\bibliography{LMC}

\end{document}